\documentclass{article}
\usepackage{amsmath}
\usepackage{latexsym}
\usepackage{mathrsfs}
\usepackage{amsthm}
\usepackage{amsfonts}
\usepackage{setspace}
\usepackage{amssymb}
\usepackage{enumerate}
\usepackage{verbatim}
\usepackage{cite}
\usepackage[toc,page]{appendix}
\usepackage{geometry}
\usepackage[curve,arrow]{xy}
\newtheorem{theorem}{Theorem}[section]
\newtheorem{proposition}[theorem]{Proposition}

\newtheorem{lemma}[theorem]{Lemma}
\newtheorem{definition}[theorem]{Definition}

\newtheorem{example}[theorem]{Example}

\newtheorem{remark}[theorem]{Remark}

\DeclareMathOperator{\R}{\mathbb{R}}

\begin{document}

\title{\bf Notions of Regularity for Functions of a Split-Quaternionic Variable}

\author{John A. Emanuello\\
Florida State University\\	
\texttt{jemanuel@math.fsu.edu}\\ 
\and
 Craig A. Nolder\\
Florida State University\\	
\texttt{nolder@math.fsu.edu}\\
}


\maketitle

\begin{abstract}
	Notions of a ``holomorphic" function theory for functions of a split-quaternionic variable have been of recent interest. We describe two found in the literature and show that one notion encompasses a small class of functions, while the other gives a richer collection. In the second instance, we describe a simple subclass of functions and give two examples of an analogue of the Cauchy-Kowalewski extension in this context.
\end{abstract}

\section{Introduction}
One need look no further than a text on complex analysis, such as \cite{Ahl} or especially \cite{JonesSingerman}, to know that algebraic properties of $\mathbb{C}$ play a major role in the analysis and geometry of the plane. The simple fact that $i^2=-1$ gives rise to the Cauchy-Riemann equations, which is the foundation of the theory of holomorphic functions, which are those functions of a complex variable which are differentiable in a complex sense. Indeed,  the existence of the limit of the difference quotient
$$\lim_{\Delta z \to 0}\frac{f(z+\Delta z)-f(z)}{\Delta z}$$
means that the limit is the same whether $\Delta z=\Delta x$ or $\Delta z=i\Delta y$. That is,
$$\frac{\partial u}{\partial x}+i\frac{\partial v}{\partial x}=\frac{1}{i}\frac{\partial u}{\partial y}+\frac{\partial v }{\partial y},$$
and the C-R equations are obtained:
$$\frac{\partial u}{\partial x}=\frac{\partial v}{\partial y} \text{ and } \frac{\partial v}{\partial x}=-\frac{\partial u}{\partial y}.$$
The minus sign in the second equation occurs because $\frac{1}{i}=-i$, which is a direct consequence of $i^2=-1$. Thus, when a function of a complex variable with $C^1$ components is holomorphic if and only if the C-R equations are satisfied.

One may also consider functions of a complex variable which are annihilated by the operator
$$\partial_{\bar{z}}:=\frac{1}{2}\left(\frac{\partial}{\partial x}+i\frac{\partial}{\partial y} \right). $$
Indeed, a $C^1$ function is holomorphic if and only if it is annihilated by $\partial_{\bar{z}}$ and its complex derivative is given by $\partial_{z}f$, where 
$$\partial_{z}:=\frac{1}{2}\left(\frac{\partial}{\partial x}-i\frac{\partial}{\partial y} \right).$$

Unlike in the complex case, when we consider functions of a split-quaternionic variable and explore the two analogous ways of defining a holomorphic function, we find that they are not equivalent. Thus, two different theories of holomorphic functions can be studied, as in \cite{Libine1,MasEtAl}. However, the one in \cite{Libine1} stands out as the more natural analogue because it gives rise to a (relatively) large class of functions to be studied. Indeed, for the analogue defined in \cite{MasEtAl} we show (by adopting a proof of an analogous statement in Sudbery's paper \cite{Sudbery}) that only affine functions, which is a (relatively) small class of functions, satisfy the given conditions.
\subsection{The Split-Quaternions}

The split-quaternions are the real Clifford algebra 
$$C\ell_{1,1}:= \left\lbrace Z=x_0+x_1i+x_2j +x_3ij\ : \ x_0,x_1, x_2, x_3 \in \mathbb{R} \right\rbrace. $$ 

Functions of a split-quaternionic variable and notions of regularity have been the subject of interest in the literature \cite{Libine1,MasEtAl}. It is worth noting that the split-quaternions contain both the complex and split-complex numbers as subalgebras.

In a manner similar to the split-complex case \cite{EmanNold,Kisil97c}, we may obtain the indefinite quadratic form $Q_{2,2}$ by
$$Z\overline{Z}=x^2_0+x^2_1-x^2_2-x^2_3.$$
Hence we shall identify the split-quaternions with $\mathbb{R}^{2,2}$.

There are a number of ways to express the split-quaternions as $2 \times 2$ matrices over $\mathbb{R}$ and $\mathbb{C}$.

\begin{lemma}\label{matrix}
	As algebras, the split-quaternions and real $2\times 2$ matrices are isomorphic.
\end{lemma}
\begin{proof}
	If we identify
	$$1\sim  \left[ \begin {array}{cc} 1&0\\ \noalign{\medskip}0&1\end {array}
	\right], \ i\sim  \left[ \begin {array}{cc} 0&-1\\ \noalign{\medskip}1&0\end {array}
	\right], \ j\sim \left[ \begin {array}{cc} 0&1\\ \noalign{\medskip}1&0\end {array}
	\right],$$
	then we may map $C\ell_{1,1}$  to the real $2\times 2$ matrices by
	$$x_0+x_1i+x_2j +x_3ij \longmapsto  \left[ \begin {array}{cc} x_{{0}}+x_{{3}}&-x_{{1}}+x_{{2}}
	\\ \noalign{\medskip}x_{{1}}+x_{{2}}&x_{{0}}-x_{{3}}\end {array}
	\right]. $$ 
	Notice that $$\det\left[ \begin {array}{cc} x_{{0}}+x_{{3}}&-x_{{1}}+x_{{2}}
	\\ \noalign{\medskip}x_{{1}}+x_{{2}}&x_{{0}}-x_{{3}}\end {array}
	\right]=x^{2}_{0} +x^{2}_{1}-x^{2}_{2}-x^{2}_{3},$$
	which is the form $Q_{2,2}$. It's easy to check that this gives an algebra homomorphism. 
	Further,
	$$\left[ \begin {array}{cc} y_1 & y_2
	\\ \noalign{\medskip}y_3 & y_4\end {array}
	\right] \longmapsto \frac{1}{2}\left[\left( y_1+y_4\right)+\left(y_3-y_2 \right)i+\left(y_3+y_2 \right)j+\left(y_1-y_4 \right)ij  \right] $$
	gives a two-sided inverse, so that the above is an algebra isomorphism.
\end{proof}

\section{Notions of Holomorphic}
The functions we are concerned with are 
$$f: U\subseteq \mathbb{R}^{2,2} \to C\ell_{1,1},$$
where $U$ is open (in the euclidean sense). As higher dimensional analogues of functions of a complex variable, we are interested in obtaining an analogous definition for \emph{holomorphic}. As we shall see, there are various ways of doing this in the literature.

The first and most interesting way is through split quaternionic valued differential operators \cite{Libine1}. The second is more recent and less interesting and is obtained by considering a difference quotient \cite{MasEtAl}. 

\subsection{Analogues of the Cauchy-Riemann Operator}
Recall that in complex analysis, one considers the Dirac operators $\partial_{\overline{z}} \text{ and } \partial_{z},$
whose product (in either order) gives the Laplacian for $\mathbb{R}^2$, usually denoted by $\Delta$. Of course, $f$ is called holomorphic if $\partial_{\overline{z}}f=0$ and its (complex) derivative is given by $\partial_{z} f$. Additionally, the real and imaginary parts of $f$ are harmonic functions, and the Dirichlet problem is well-posed.

The question asked in the literature is: \emph{Can we define operators valued in $C\ell_{1,1}$ which resemble $\partial_{\overline{z}}$ and $\partial_{z}$?} This question has been answered in the affirmative, although with little mention of the differential geometry which lies just below the surface. 

However the question we are really asking is: can we factorize the Laplacian in $\mathbb{R}^{2,2}$ with linear first order operators over  $C\ell_{1,1}$? In this semi-Riemannian manifold, the Laplacian, which is understood to be the derivative of the gradient, is given by \cite{Oneill}:
$$\Delta_{2,2}={\partial^2 \over \partial x^2_0} +  {\partial^2 \over \partial x^2_1} -  {\partial^2 \over \partial x^2_2} -  {\partial^2 \over \partial x^2_3}.$$
It is easy to check that the linear operators
\begin{align*}
&\overline{\partial} := {\partial \over \partial x_0}\ + \ i{\partial \over \partial x_1}\ - \ j{\partial \over \partial x_2}\ - \ ij{\partial \over \partial x_3} \text{ and }\\
&{\partial} := {\partial \over \partial x_0}\ - \ i{\partial \over \partial x_1}\ + \ j{\partial \over \partial x_2}\ + \ ij{\partial \over \partial x_3}
\end{align*}
are factors of $\Delta_{2,2}$. Due to the non-commutativity of $C\ell_{1,1}$, these operators may be applied to functions on either the left or right and with different results, in general.

\begin{remark}
	There are other factorizations of $\Delta_{2,2}$ inside $C\ell_{1,1}$. Our choice of $\overline{\partial}$ is deliberate-- it is the gradient inside the semi-Riemannian manifold $\mathbb{R}^{2,2}$. For alternative interpretation of this idea, see \cite{Obol}.
\end{remark}

\begin{definition}
	Let $U\subset C\ell_{1,1} \cong \mathbb{R}^{2,2}$ and let $F: U \to C\ell_{1,1}$ be $C^{1}\left(U\right)$. We say $F$ is left regular if
	$$\overline{\partial} F =0$$
	for every $Z \in U$. Similarly, we say $F$ is right regular if
	$$F \overline{\partial} =0$$
	for every $Z \in U$.
\end{definition}

We have adopted the above definition from \cite{Libine1}, which contains a proof of a Cauchy-like integral formula for left-regular functions.

By multiplying arbitrary $F$ with $\overline{\partial}$ we obtain the following conditions which make it easier to check left and right regularity.
\begin{proposition}
	Let $F: U \to C\ell_{1,1}$ be $C^{1}\left(U\right)$. Then $F$ is left regular if and only if it satisfies the system of PDEs:
	$$ \begin{cases} \displaystyle
	{\partial f_0 \over \partial x_0}- {\partial f_1 \over \partial x_1}-{\partial f_2 \over \partial x_2}-{\partial f_3 \over \partial x_3}=0 \\ \\
	\displaystyle {\partial f_1 \over \partial x_0}+ {\partial f_0 \over \partial x_1}+{\partial f_3 \over \partial x_2}-{\partial f_2 \over \partial x_3}=0 \\ \\
	\displaystyle
	{\partial f_2 \over \partial x_0}- {\partial f_3 \over \partial x_1}-{\partial f_0 \over \partial x_2}-{\partial f_1 \over \partial x_3}=0 \\ \\
	\displaystyle {\partial f_3 \over \partial x_0}+ {\partial f_2 \over \partial x_1}+{\partial f_1 \over \partial x_2}-{\partial f_0 \over \partial x_3}=0.
	\end{cases}$$
\end{proposition}

\begin{proof}
	The proof follows directly from the definition. Simply multiply in the proper order, collect like components together, and equate them to zero to obtain the desired system.
\end{proof}

\begin{proposition}
	Let $F: U \to C\ell_{1,1}$ be $C^{1}\left(U\right)$. Then $F$ is right regular if and only if it satisfies the system of PDEs:
	$$ \begin{cases} \displaystyle
	{\partial f_0 \over \partial x_0}- {\partial f_1 \over \partial x_1}-{\partial f_2 \over \partial x_2}-{\partial f_3 \over \partial x_3}=0 \\ \\
	\displaystyle {\partial f_1 \over \partial x_0}+ {\partial f_0 \over \partial x_1}-{\partial f_3 \over \partial x_2}+{\partial f_2 \over \partial x_3}=0 \\ \\
	\displaystyle
	{\partial f_2 \over \partial x_0}+ {\partial f_3 \over \partial x_1}-{\partial f_0 \over \partial x_2}+{\partial f_1 \over \partial x_3}=0 \\ \\
	\displaystyle {\partial f_3 \over \partial x_0}- {\partial f_2 \over \partial x_1}-{\partial f_1 \over \partial x_2}-{\partial f_0 \over \partial x_3}=0.
	\end{cases}$$
\end{proposition}

However, this notion of regularity is also some what unsatisfying, for simple analogues of holomorphic functions in the complex plane are not regular. 

\begin{example}
	Let $A=a+ib+jc+ijd \in C\ell_{1,1}$. Then
	\begin{align*}
	AZ=&(ax_0-bx_1+cx_2+dx_3)+i(bx_0+ax_1+dx_2-cx_3)\\
	+j&(cx_0+dx_1+ax_2-bx_3)+ij(dx_0-cx_1+bx_2+ax_3).
	\end{align*}
	
	Thus,
	\begin{align*}
	\overline{\partial} (AZ)&= (a+ib+jc+dij)+i(-b+ai+dj-cij)\\
	&-j(c+id+aj+bij)-ij(d-ic-bj+aij)\\
	&=-2a+i2b+j2c+ij2d\\
	&=-2\overline{A},
	\end{align*}	
	A similar calculation shows that
	$$ \left( AZ\right)\overline{\partial}=-2A.$$
	
	Other calculations show that the function $ZA$ is also neither left-regular nor right-regular.
\end{example}

We obtain similar systems of PDEs if we consider the equations ${\partial}F=0$ and $F{\partial}=0$:
$$ \begin{cases} \displaystyle
{\partial f_0 \over \partial x_0}+ {\partial f_1 \over \partial x_1}+{\partial f_2 \over \partial x_2}+{\partial f_3 \over \partial x_3}=0 \\ \\
\displaystyle {\partial f_1 \over \partial x_0}- {\partial f_0 \over \partial x_1}-{\partial f_3 \over \partial x_2}+{\partial f_2 \over \partial x_3}=0 \\ \\
\displaystyle
{\partial f_2 \over \partial x_0}+ {\partial f_3 \over \partial x_1}-{\partial f_0 \over \partial x_2}+{\partial f_1 \over \partial x_3}=0 \\ \\
\displaystyle {\partial f_3 \over \partial x_0}- {\partial f_2 \over \partial x_1}-{\partial f_1 \over \partial x_2}+{\partial f_0 \over \partial x_3}=0
\end{cases}$$
and 
$$ \begin{cases} \displaystyle
{\partial f_0 \over \partial x_0}+ {\partial f_1 \over \partial x_1}+{\partial f_2 \over \partial x_2}+{\partial f_3 \over \partial x_3}=0 \\ \\
\displaystyle {\partial f_1 \over \partial x_0}- {\partial f_0 \over \partial x_1}+{\partial f_3 \over \partial x_2}-{\partial f_2 \over \partial x_3}=0 \\ \\
\displaystyle
{\partial f_2 \over \partial x_0}- {\partial f_3 \over \partial x_1}+{\partial f_0 \over \partial x_2}-{\partial f_1 \over \partial x_3}=0 \\ \\
\displaystyle {\partial f_3 \over \partial x_0}+ {\partial f_2 \over \partial x_1}+{\partial f_1 \over \partial x_2}+{\partial f_0 \over \partial x_3}=0.
\end{cases}$$

These also produce unsatisfying analogues of holomorphic since linear functions, again, fail these conditions. 

\begin{example}
	Let $A=a+ib+jc+ijd \in C\ell_{1,1}$. Then,
	\begin{align*}
	{\partial}(AZ)&= (a+ib+jc+dij)-i(-b+ai+dj-cij)\\
	&+j(c+id+aj+bij)+ij(d-ic-bj+aij)\\
	&=4a.
	\end{align*}	
	A similar calculation shows that
	$$ \left( AZ\right){\partial}=4A.$$
	Other calculations show that the function $ZA$ is not annihilated by ${\partial}$ on either side.
\end{example}

\subsection{Difference Quotients}
Recall, another (and probably primary) way to define holomorphic functions is via the limit of a difference quotient:
$$ \lim_{\Delta z \to 0}{f\left( z+\Delta z\right)-f\left( z\right) \over  \Delta z}.$$
One obtains the Cauchy-Riemann equations by allowing $\Delta z$ to approach $0$ along the real axis and again along the imaginary axis and then setting the results equal to each other.

In Masouri et. al., a similar method is used to produce another analogue of holomorphic \cite{MasEtAl}. However, since the split-quaternions are not commutative, so there are two ways to construct an analogue of the difference quotient. In Masouri the ``quotient" is defined by
$$\lim_{\Delta Z \to 0}\left( f\left( Z+\Delta Z\right)-f\left( Z\right)\right)   \left( \Delta Z\right)^{-1}.$$
When this limit exists, such functions are called right $C\ell_{1,1}$-differentiable. By setting $\Delta Z$ equal to $\Delta x_0$, $ i\Delta x_1$, $ j\Delta x_2$, and $ ij\Delta x_3$, taking the limit in each instance, we get four ways to take the ``derivative'' \cite{MasEtAl}. That is,

$$\lim_{\Delta x_0 \to 0}\left( f\left( \zeta+\Delta x_0\right)-f\left( \zeta\right)\right)   \left( \Delta x_0\right)^{-1}=\frac{\partial f_0}{\partial x_0}+i\frac{\partial f_1}{\partial x_0}+j\frac{\partial f_2}{\partial x_0}+ij\frac{\partial f_3}{\partial x_0},$$

$$\lim_{i\Delta x_1 \to 0}\left( f\left( \zeta+\Delta x_0\right)-f\left( \zeta\right)\right)   \left( i\Delta x_1\right)^{-1}=-i\frac{\partial f_0}{\partial x_1}+\frac{\partial f_1}{\partial x_1}+ij\frac{\partial f_2}{\partial x_1}-j\frac{\partial f_3}{\partial x_1},$$

$$\lim_{j\Delta x_2 \to 0}\left( f\left( \zeta+\Delta x_0\right)-f\left( \zeta\right)\right)   \left( j\Delta x_2\right)^{-1}=j\frac{\partial f_0}{\partial x_2}+ij\frac{\partial f_1}{\partial x_2}+\frac{\partial f_2}{\partial x_2}+i\frac{\partial f_3}{\partial x_2},$$
and
$$\lim_{ij\Delta x_3 \to 0}\left( f\left( \zeta+\Delta x_0\right)-f\left( \zeta\right)\right)   \left( ij\Delta x_3\right)^{-1}=ij\frac{\partial f_0}{\partial x_3}-j\frac{\partial f_1}{\partial x_3}-i\frac{\partial f_2}{\partial x_3}+\frac{\partial f_3}{\partial x_3}.$$

Equating the four results, we obtain the system of PDEs \cite{MasEtAl}:

$$ \begin{cases} \displaystyle
{\partial f_0 \over \partial x_0}= {\partial f_1 \over \partial x_1}={\partial f_2 \over \partial x_2}= {\partial f_3 \over \partial x_3} \\ \\
\displaystyle {\partial f_1 \over \partial x_0}=- {\partial f_0 \over \partial x_1}= {\partial f_3 \over \partial x_2}= -{\partial f_2 \over \partial x_3} \\ \\
\displaystyle
{\partial f_2 \over \partial x_0}= - {\partial f_3 \over \partial x_1}={\partial f_0 \over \partial x_2}=-{\partial f_1 \over \partial x_3} \\ \\
\displaystyle {\partial f_3 \over \partial x_0}= {\partial f_2 \over \partial x_1}={\partial f_1 \over \partial x_2}={\partial f_0 \over \partial x_3}.
\end{cases}$$

Although the work which introduces this notion of differentiability, \cite{MasEtAl}, does not mention any specific examples of functions of right $C\ell_{1,1}$-differentiable functions, an entire class of functions can be easily shown to have this property.
\begin{example}
	Recall that 
	\begin{align*}
	AZ+K=&(ax_0-bx_1+cx_2+dx_3+k)+i(bx_0+ax_1+dx_2-cx_3+\ell)\\
	+j&(cx_0+dx_1+ax_2-bx_3+m)+ij(dx_0-cx_1+bx_2+ax_3+n).
	\end{align*}
	Notice $f(Z)=AZ+K$ is right $C\ell_{1,1}$-differentiable. Indeed, the ``derivative'' is
	
	$$\lim_{\Delta Z \to 0}\left( A \left(  Z+\Delta Z \right)+K -AZ-K\right)   \left( \Delta Z\right)^{-1}=\lim_{\Delta Z \to 0}\left( A\Delta Z\right)\left( \Delta Z\right)^{-1}=A.$$
\end{example}

\begin{theorem}
	Let $F:U\subseteq \R^{2,2}\to C\ell_{1,1}$. Then $F$ is right $C\ell_{1,1}$-differentiable if and only $F(Z)=AZ+K$, where $A,K\in C\ell_{1,1}$. That is, the right $C\ell_{1,1}$-differentiable functions must be affine mappings.
\end{theorem}

\begin{proof}
	As similar fact is true for functions of a quaternionic variable and so we follow a similar proof from Sudbery's paper\footnote{We are very grateful to Professor Uwe K\"{a}hler of University of Aveiro for bringing this paper to our attention.} \cite{Sudbery}.
	
	First notice that $Z=(x_0+x_1i)+(x_2+x_3i)j=z+wj$. As such we may write $f(Z)=g(z,w)+h(z,w)j$, where $g(z,w)=f_0(z,w)+if_1(z,w)$ and $h(z,w)=f_2(z,w)+if_3(z,w)$.
	
	Now, the above system of PDEs gives us that $g$ is holomorphic with respect to the complex variables $z$ and $\overline{w}$. Similarly, $h$ is holomorphic with respect to the complex variables $w$ and $\overline{z}$. Additionally,
	\begin{align*}
	&\frac{\partial g}{\partial z}= \frac{\partial f_0}{\partial x_0}+i\frac{\partial f_1}{\partial x_0} = \frac{\partial f_2}{\partial x_2}+i\frac{\partial f_3}{\partial x_2}=\frac{\partial h}{\partial w},\\
	&\frac{\partial g}{\partial \overline{w}}= -\frac{\partial f_1}{\partial x_3}+i\frac{\partial f_0}{\partial x_3}= -\frac{\partial f_3}{\partial x_1}+i\frac{\partial f_2}{\partial x_1}=\frac{\partial h}{\partial \overline{z}}.
	\end{align*}
	Now, $g$ and $h$ have continuous partial derivatives of all orders. Thus, we must have
	\begin{align*}
	&\frac{\partial^2 g}{\partial z^2}=\frac{\partial }{\partial z}\left( \frac{\partial h}{\partial w}\right) = \frac{\partial }{\partial w}\left( \frac{\partial h}{\partial z}\right)=0,\\
	&\frac{\partial^2 h}{\partial w^2}=\frac{\partial }{\partial w}\left( \frac{\partial g}{\partial z}\right) = \frac{\partial }{\partial z}\left( \frac{\partial g}{\partial w}\right)=0,\\
	&\frac{\partial^2 g}{\partial \overline{w}^2}=\frac{\partial }{\partial \overline{w}}\left( \frac{\partial h}{\partial \overline{z}}\right) = \frac{\partial }{\partial \overline{z}}\left( \frac{\partial h}{\partial \overline{w}}\right)=0,\\
	&\frac{\partial^2 h}{\partial \overline{z}^2}=\frac{\partial }{\partial \overline{z}}\left( \frac{\partial g}{\partial \overline{w}}\right) = \frac{\partial }{\partial \overline{w}}\left( \frac{\partial g}{\partial \overline{z}}\right)=0.
	\end{align*}
	W.L.O.G. we may assume that $U$ is connected  and convex (since each connected component may be covered by convex sets, which overlap pair-wise on convex sets). Thus integrating on line segments allows us to conclude that $g$ and $h$ are linear:
	\begin{align*}
	& g(z,w)= \alpha +\beta z+\gamma \overline{w}+\delta z\overline{w},\\
	& h(z,w)= \epsilon +\eta \overline{z}+\theta w+\nu \overline{z}w.
	\end{align*}
	Since $\frac{\partial g}{\partial z}=\frac{\partial h}{\partial w}$, we must have that $\beta=\theta$ and $\delta=\nu=0$. Also since $\frac{\partial g}{\partial \overline{w}}=\frac{\partial h}{\partial \overline{z}}$, it is the case that $\gamma=\eta$. Thus,
	\begin{align*}
	f(Z)&=g(z,w)+h(z,w)j\\
	&=(\alpha +\beta z+\gamma \overline{w})+(\epsilon+ \gamma \overline{z}+\beta w)j\\
	&=(\beta + \gamma j) (z+wj)+ (\alpha+\epsilon j)\\
	&=A Z +K,
	\end{align*}
	as required.
\end{proof}

\begin{remark}
	The above theorem proves that right $C\ell_{1,1}$-differentiable functions are not left or right regular and conversely (except for when $A=0$). Indeed, they are also not annihilated by ${\partial}$ on either side.
\end{remark}

As an alternative to the definition found in \cite{MasEtAl}, one may reverse the multiplication in the difference quotient to obtain 
$$\lim_{\Delta Z \to 0}\left( \Delta Z\right)^{-1} \left( f\left( Z+\Delta Z\right)-f\left( Z\right)\right).$$
When this limit exists, such functions are called left $C\ell_{1,1}$-differentiable. Proceeding as above, a slightly different system of PDEs than the one found in \cite{MasEtAl} is obtained:

$$ \begin{cases} \displaystyle
{\partial f_0 \over \partial x_0}= {\partial f_1 \over \partial x_1}={\partial f_2 \over \partial x_2}= {\partial f_3 \over \partial x_3} \\ \\
\displaystyle {\partial f_1 \over \partial x_0}=- {\partial f_0 \over \partial x_1}= -{\partial f_3 \over \partial x_2}= {\partial f_2 \over \partial x_3} \\ \\
\displaystyle
{\partial f_2 \over \partial x_0}= {\partial f_3 \over \partial x_1}={\partial f_0 \over \partial x_2}={\partial f_1 \over \partial x_3} \\ \\
\displaystyle {\partial f_3 \over \partial x_0}= -{\partial f_2 \over \partial x_1}=-{\partial f_1 \over \partial x_2}={\partial f_0 \over \partial x_3}.
\end{cases}$$

\begin{example}
	Recall that 
	\begin{align*}
	AZ=&(ax_0-bx_1+cx_2+dx_3)+i(bx_0+ax_1+dx_2-cx_3)\\
	+j&(cx_0+dx_1+ax_2-bx_3)+ij(dx_0-cx_1+bx_2+ax_3).
	\end{align*}
	Notice $f(Z)=AZ+K$ is not left $C\ell_{1,1}$-differentiable.
	
	However, the map $F(Z)=ZA+K$ is left $C\ell_{1,1}$-differentiable. Indeed, the ``derivative'' is
	
	$$\lim_{\Delta Z \to 0}\left( \Delta Z\right)^{-1}\left(  \left(  Z+\Delta Z \right)A +K -ZA-K\right) =\lim_{\Delta Z \to 0}\left( \Delta Z\right)^{-1}\left( \Delta Z A \right)=A.$$
\end{example}

\begin{theorem}
	Let $F:U\subseteq \R^{2,2}\to C\ell_{1,1}$. Then $F$ is left $C\ell_{1,1}$-differentiable if and only $F(Z)=ZA+K$, where $A,K\in C\ell_{1,1}$. That is, the left $C\ell_{1,1}$-differentiable functions must be affine mappings.
\end{theorem}

\begin{proof}
	We can make a few adjustments to the proof of the right $C\ell_{1,1}$-differentiable case.
	
	First note that if we write $F(Z)=g(z,w)+jh(z,w)$, where $g(z,w)=f_0(z,w)+if_1(z,w)$, $h(z,w)=f_2(z,w)-if_3(z,w)$, and the complex variables $z,w$ as above.
	
	The system of PDEs above assures that $g$ is holomorphic with respect to $z$ and $w$, while $h$ is holomorphic with respect to $\overline{z}$ and $\overline{w}$. Additionally, we get that
	\begin{align*}
	&\frac{\partial g}{\partial z}=\frac{\partial h}{\partial \overline{w}}\\
	&\frac{\partial g}{\partial w}=\frac{\partial h}{\partial \overline{z}}.
	\end{align*}
	
	We also have that $g$ and $h$ have partial derivatives of all orders and similarly to the ``right'' case the second partials vanish:
	$$\frac{\partial^2 g}{\partial z^2}=
	\frac{\partial^2 h}{\partial \overline{w}^2}
	=\frac{\partial^2 g}{\partial w^2}=\frac{\partial^2 h}{\partial \overline{z}^2}=0.$$
	
	Thus, by the same argument for the right $C\ell_{1,1}$-differentiable proof, we conclude that $g$ and $h$ are linear:
	\begin{align*}
	& g(z,w)= \alpha +\beta z+\gamma w+\delta zw,\\
	& h(z,w)= \epsilon +\eta \overline{z}+\theta \overline{w}+\nu \overline{z}\overline{w}.
	\end{align*}
	
	Since $\frac{\partial g}{\partial z}=\frac{\partial h}{\partial \overline{w}}$ and  $\frac{\partial g}{\partial w}=\frac{\partial h}{\partial \overline{z}}$, we must have that $\beta=\theta$, $\gamma=\eta$,  and $\delta=\nu=0$. Thus,
	\begin{align*}
	f(Z)&=g(z,w)+jh(z,w)\\
	&=(\alpha +\beta z+\gamma w )+j(\epsilon+ \gamma \overline{z}+\beta \overline{w})\\
	&=(z+wj)(\beta + \overline{\gamma} j) + (\alpha+\overline{\epsilon} j)\\
	&=ZA +K,
	\end{align*}
	as required.
\end{proof}

\begin{remark}
	Thus, right $C\ell_{1,1}$-differentiability is perhaps not a good analogue of holomorphic. Even though these are equivalent notions in the complex setting, in the split quaternionic setting there are more directions in which to take the limit and this requires much stronger conditions. For this reason we are justified in studying functions in the kernels of the operators, and not the $C\ell_{1,1}$-differentiable functions.
\end{remark}

\subsection{Regularity and John's Equation}
Given a $F: U \to C\ell_{1,1}$ whose components are at least $C^2$ and which satisfies at least one of the following:
$$\overline{\partial} F=0, \ F\overline{\partial}=0, \ {\partial}F=0, \text{ or } F{\partial}=0,$$
must have components which satisfy John's equation \cite{Libine1}:
$$\Delta_{2,2} u=0.$$
Such functions are said to be ultra-hyperbolic.

In fact, we can use ultra-hyperbolic functions to build regular functions.

\begin{theorem}
	Let $f: U\to \mathbb{R}$ be ultra-hyperbolic, then ${\partial}f$ is both left and right regular.
\end{theorem}

\begin{proof}
	Write $F={\partial} f$. Then clearly
	$$\overline{\partial} F= \Delta_{2,2}f=0=({\partial} f)\overline{\partial}=F\overline{\partial}.$$
\end{proof}

It turns out that left and right differentiable functions also have components which are ultra-hyperbolic.

\begin{theorem}
	Let $F: U \to C\ell_{1,1} $, with components which are at least $C^2$, be left-differentiable or right-differentiable. Then the components of $F$ are ultra-hyperbolic. 
\end{theorem}

\begin{proof}
	Suppose $F(x_0,x_1,x_2,x_3)=f_0+f_1i +f_2j+f_3ij$ is right-differentiable. Then notice that
	\begin{align*}
	{\partial^2 f_0 \over \partial x^2_0} &+  {\partial^2 f_0 \over \partial x^2_1} -  {\partial^2 f_0 \over \partial x^2_2} -  {\partial^2 f_0\over \partial x^2_3}\\
	&={\partial \over \partial x_0}\left({\partial f_1 \over \partial x_1} \right) +{\partial \over \partial x_1}\left(-{\partial f_1 \over \partial x_0} \right) -{\partial \over \partial x_2}\left(-{\partial f_1 \over \partial x_3} \right)-{\partial \over \partial x_3}\left({\partial f_1 \over \partial x_2} \right)\\
	&=0.
	\end{align*}
	A similar argument works for the other $f_i$ and for the left-differentiable case.
\end{proof}



\section{A Theory of Left-Regular Functions}
With all of these notions of holomorphic functions, it becomes necessary to choose one and deem it the ``canonical'' one. Since the difference quotients do not yield an extensive class of functions, we believe use of an operator to be the be the best place to start. Given the association between $C\ell_{1,1}$ and $\mathbb{R}^{2,2}$, it seems $\overline{\partial}$ is the ideal operator for our purposes, since it is also the gradient in $\mathbb{R}^{2,2}$ (and since it is an analogue of $\partial_{\bar{z}}$, which is $\frac{1}{2}$ times the gradient of $\R^2$). Given the overwhelming convention of applying operators on the left of functions, we choose left-regular to be the canonical notion of holomorphic.

Indeed, this is the one chosen by Libine \cite{Libine1}. In his work, he shows that left-regular functions satisfy a Cauchy-like integral formula.

\begin{theorem}[Libine's Integral Formula]
	Let $U\subseteq C\ell_{1,1}$ be a bounded open (in the Euclidean topology) region with smooth boundary $\partial U$. Let $f: U \to C\ell_{1,1}$ be a function which extends to a real-differentiable function on an open neighborhood $V\subseteq C\ell_{1,1}$ of $\overline{U}$ such that $\overline{\partial} f=0$. Then for any $Z_0 \in C\ell_{1,1}$ such that the boundary of $U$ intersects the cone $C=\left\{Z\in C\ell_{1,1}: (Z-Z_0)\overline{(Z-Z_0)}=0 \right\}$ transversally, we have
	\begin{align*}
	\lim_{\epsilon \to 0}\frac{-1}{2\pi^2}\int_{\partial U}\frac{\overline{\left(Z-Z_0\right)}}{(Z-Z_0)\overline{(Z-Z_0)}+i\epsilon\left\|Z-Z_0\right\|^2}\cdot& dZ \cdot f(Z)\\
	&=	\begin{cases}
	f(Z_0) &\text{if $Z_0\in U$} \\
	0 &\text{else}
	\end{cases},
	\end{align*}
	where the three form $dZ$ is given by
	$$dZ=dx_1\wedge dx_2\wedge dx_3-(dx_0\wedge dx_2\wedge dx_3)i+(dx_0\wedge dx_1\wedge dx_3)j-(dx_0\wedge dx_1\wedge dx_2)ij.$$
\end{theorem}

Given this interesting property has been proven, it is some what surprising that a more detailed description of left-regular functions has not been given in the literature. So we conclude by showing that some left regular functions have a simple description.

\subsection{A Class of Left Regular Functions}

To date, the author has not been able to find a description for left regular functions in a manner similar to the split-complex case \cite{EmanNold, Libine2}. It may be the case that no such description exists in general. However, it is possible to give a large class of left-regular functions a simple description.

\begin{theorem}\label{oneclass}
	Let $F: U\subseteq \mathbb{R}^{2,2} \to C\ell_{1,1}$ have the form
	\begin{align*}
	F(x_0,x_1,x_2,x_3)&=\left( g_1(x_0+x_2,x_1+x_3)+g_2(x_0-x_2, x_1-x_3)\right) \\
	&+\left( g_3(x_0-x_2, x_1-x_3)+g_4(x_0+x_2,x_1+x_3)\right)i\\
	&+\left( g_1(x_0+x_2,x_1+x_3)-g_2(x_0-x_2, x_1-x_3)\right)j \\
	&+\left( g_3(x_0-x_2, x_1-x_3)-g_4(x_0+x_2,x_1+x_3)\right)ij,
	\end{align*}
	where $g_i\in C^1(U)$. Then $\overline{\partial} F=0$. 
\end{theorem}

\begin{proof}
	We can easily check that such an $F$ satisfies the necessary system of PDEs. However, it is far more enlightening to see how one can arrive at such a solution.
	
	Write $F=f_0+f_1 i+f_2 j+f_3 ij$.  Using an argument from \cite{EmanNold}, we have that
	\begin{align*}
	\overline{\partial} &=\left( {\partial \over \partial x_0}\ - \ j{\partial \over \partial x_2}\right) \ + \ i\left( {\partial \over \partial x_1}\  - \ j{\partial \over \partial x_3}\right)\\
	&=2\left({\partial \over \partial v_0}j_+ +{\partial \over \partial u_0}j_- \right)\ + \ 2i\left({\partial \over \partial v_1}j_+ +{\partial \over \partial u_1}j_-  \right) \\
	&:= \partial_1\ + \ i\partial_2,
	\end{align*}
	
	where $u_0=x_0+x_2$, $v_0=x_0-x_2$, $u_1=x_1+x_3$, $u_1=x_1-x_3$, $j_+=\displaystyle \frac{1+j}{2}$, and $j_-=\displaystyle \frac{1-j}{2}$. 	
	The key fact is that $j_+$ and $j_-$ are idempotents and annihilate each other. Also, notice that $ij_+=j_-i$ and $ij_-=j_+i$.
	
	Similarly, we may write 
	$$F=\left(F_0 j_+ +F_1 j_- \right) \ + \ i\left(F_2 j_+ +F_3 j_- \right).$$
	
	Now, \emph{one way} in which $\overline{\partial} F=0$ is if 
	\begin{align*}
	\partial_1\left(F_0 j_+ +F_1 j_- \right)&=\partial_2\left(F_0 j_+ +F_1 j_- \right)\\
	&=\partial_1\left(i\left(F_2 j_+ +F_3 j_- \right) \right)=\partial_2\left(i\left(F_2 j_+ +F_3 j_- \right) \right)=0.
	\end{align*}
	
	Using the above facts about $j_+$ and $j_-$, we see that the conditions implies that
	$$ \begin{cases} \displaystyle
	{\partial F_0 \over \partial v_0}= {\partial F_0 \over \partial v_1}=0\\ \\
	\displaystyle {\partial F_1 \over \partial u_0}= {\partial F_1 \over \partial u_1}=0 \\ \\
	\displaystyle
	{\partial F_2 \over \partial u_0}= {\partial F_2 \over \partial u_1}=0 \\ \\
	\displaystyle {\partial F_3 \over \partial v_0}= {\partial F_3 \over \partial v_1}=0
	\end{cases}$$
	
	This, of course, means that
	$$F_0=F_0(u_0,u_1), \ F_1=F_1(v_0,v_1),\ F_2=F_2(v_0,v_1), \ F_3=F_3(u_0,u_1).$$
	
	Translating back to the original coordinates, we see $F$ has the desired form.
	
\end{proof}

The converse is not true, in general. Here is a simple counter-example.

\begin{example}\label{counterexample}
	Consider the $C\ell_{1,1}$-valued function
	$$f(x_0,x_1,x_2,x_3)=x_1x_2x_3-x_0x_2x_3i+x_0x_1x_3j+x_0x_1x_2ij.$$
	It is easy to check that $f$ satisfies the necessary system of PDEs so that $\overline{\partial} f=0$. However, notice that if we write $f$ as in the above proof, then 
	$$f=\frac{\left(u^{2}_{1}-v^{2}_{1} \right)}{4}\left(u_0j_+-v_0j_- \right) +i\frac{\left(u^{2}_{0}v^{2}_{0} \right)}{4}\left(-u_0j_++v_0j_- \right). $$
	Now,
	$$\partial_2\left[\frac{\left(u^{2}_{1}-v^{2}_{1} \right)}{4}\left(u_0j_+-v_0j_- \right) \right]= -2v_1u_0j_+-2u_1v_0\not\equiv 0.$$
	Thus, $f$ is not of the form as prescribed in Theorem \ref{oneclass}.
\end{example}

\subsection{Generating Left Regular Functions}
In a manner similar to the $C\ell_{0,n}$ case, we can also take a $C\ell_{1,1}$-valued function whose components are real analytic and generate a left regular function valued in $C\ell_{1,1}$. In fact, there are two ways to do this. The first borrows heavily from a result found in Brackx, Delanghe, and Sommen's book \cite{BrackxDelangheSommen1}.

\begin{theorem}
	Let $g(x_2, x_3)$ be a $C\ell_{1,1}$-valued function on $U\subseteq \R^2$ with real-analytic components. Then the function
	$$f(Z)=\sum_{k=0}^{\infty}{\partial}\left[ \left( \frac{x^{2k+1}_{0} + x^{2k+1}_{1} }{\left(2k+1 \right)! }\right)\Delta^k g(x_2,x_3) \right], $$
	where $\Delta$ is the Laplace operator in the $x_2x_3$-plane, is left-regular in an open neighborhood of $\left\lbrace (0,0) \right\rbrace \times U$ in $\R^{2,2}$ and $ f(0,0,x_2,x_3) =g(x_2,x_3)-ig(x_2x_3)$.
\end{theorem}

\begin{proof}
	We proceed by a similar proof found in \cite{BrackxDelangheSommen1}.
	
	Let $g(x_2,x_3)= g_0(x_2,x_3)+g_1(x_2,x_3)i+g_2(x_2,x_3)j+g_3(x_2,x_3)ij$. Since $g_{\ell}$ is analytic, then an application of Taylor's theorem gives that on every compact set $K\subset U$ there are constants $c_K$ and $\lambda_K$, depending on $K$, such that
	\begin{align*}
	&\sup_{(x_2,x_3)\in K}\left|\Delta^k g(x_2,x_3) \right|\leq (2k)!c_K\lambda^{k}_{K} \hspace{.25cm}\text{ and }\\ &\sup_{(x_2,x_3)\in K}\left|\frac{\partial}{\partial x_{\ell}}\Delta^k g(x_2,x_3) \right|\leq (2k+1)!c_K\lambda^{k}_{K},
	\end{align*}
	where $\left| \text{ }\cdot\text{ }\right|$ denotes the euclidean norm in $\R^4$. 
	
	Thus,
	\begin{align*}
	\sup_{(x_2,x_3)\in K}&\left|{\partial}\left[ \left( \frac{x^{2k+1}_{0} + x^{2k+1}_{1} }{\left(2k+1 \right)! }\right)\Delta^k g(x_2,x_3) \right] \right|\\
	&=	\sup_{(x_2,x_3)\in K}\left|\left( \frac{x^{2k}_{0}}{(2k)!}+i\frac{x^{2k}_{1}}{(2k)!}\right)\Delta^kg\right. \\
	&\left. \hspace{ 2.5cm}- \frac{x^{2k+1}_{0} + x^{2k+1}_{1} }{\left(2k+1 \right)! }\left(j \frac{\partial}{\partial x_2}\Delta^kg+ij\frac{\partial}{\partial x_3}\Delta^kg\right)  \right|\\
	&\leq \sup_{(x_0,x_1)\in K}\left[\left| \frac{x^{2k}_{0}}{(2k)!}\right|\left|\Delta^kg \right| +\left| \frac{x^{2k}_{1}}{(2k)!}\right|\left|\Delta^kg \right| \right. \\
	&\hspace{ 2.5cm}\left.+\left|\frac{x^{2k+1}_{0} + x^{2k+1}_{1} }{\left(2k+1 \right)! }  \right|\left( \left|\frac{\partial}{\partial x_2}\Delta^kg \right|+  \left|\frac{\partial}{\partial x_3}\Delta^kg \right| \right)  \right]\\
	&\leq c_{K} \left[ \left(1+2\left| x_0\right|  \right)x^{2k}_{0}\lambda^k_{K} + \left(1+2\left| x_1\right|  \right)x^{2k}_{1}\lambda^{k}_{K}\right] ,
	\end{align*}
	so that $f$ converges uniformly on
	$$\bigcup_{K\subseteq U}\left[\left(-\frac{1}{\sqrt{\lambda_K}},\frac{1}{\sqrt{\lambda_K}} \right)\times\left(-\frac{1}{\sqrt{\lambda_K}},\frac{1}{\sqrt{\lambda_K}} \right)\times \mathring{K}  \right]. $$
	
	Now,
	\begin{align*}
	\overline{\partial} f&=\sum_{k=0}^{\infty}\Delta_{2,2}\left[ \left( \frac{x^{2k+1}_{0} + x^{2k+1}_{1} }{\left(2k+1 \right)! }\right)\Delta^k g(x_2,x_3) \right]\\
	&=\sum_{k=0}^{\infty} \Delta_{2,2}\left(\frac{x^{2k+1}_{0} + x^{2k+1}_{1}}{\left(2k+1 \right)! }  \right) \Delta^k g(x_2,x_3) \\
	&\qquad +\ \sum_{k=0}^{\infty} \left(\frac{x^{2k+1}_{0} + x^{2k+1}_{1}}{\left(2k+1 \right)! } \right) \Delta_{2,2} \left(\Delta^k g(x_2,x_3)\right) \\
	&=\sum_{k=1}^{\infty}\left( \frac{x^{2k-1}_{0} + x^{2k-1}_{1}}{\left(2k-1 \right)! }\right)\Delta^k g(x_2,x_3)\\
	&\qquad-\sum_{k=0}^{\infty} \left(\frac{x^{2k+1}_{0} \ + \ x^{2k+1}_{1}}{\left(2k+1 \right)! } \right)  \Delta^{k+1} g(x_2,x_3)\\
	&=0,
	\end{align*}
	as needed.
\end{proof}

\begin{example}
	Let $g(x_2,x_3)=x_2x_3$. Then $\Delta g=0$ and the formula above gives
	\begin{align*}
	f(Z)&={\partial}\left[ (x_0+x_1)(x_2x_3)\right] \\
	&=x_2x_3-x_2x_3i+(x_0x_3+x_1x_3)j+(x_0x_2+x_1x_2).
	\end{align*}
\end{example}

A less trivial example demonstrates that the more complicated $g$ is the more complicated $f$ is.
\begin{example}
	Let $g(x_2,x_3)=x^{4}_{2}+x_2x^{3}_{3}$. Thus, $\Delta g=12x^{2}_{2}+6x_2x_3$ and $\Delta^2g=24$.

	Then from the formula, we get
	\begin{align*}
	f(Z)=&\left[ x^{4}_{0}+3x^{2}_{0}(2x^{2}_{2}+x_2x_3)+x^{4}_{2}+x_2x^{3}_{3}\right] \\
	 &\qquad +\left[x^{4}_{1}+3x^{2}_{1}(2x^{2}_{2}+x_2x_3)+x^{4}_{2}+x_2x^{3}_{3} \right] i\\
	&\qquad+\left[(x^{2}_{0}x^{2}_{1})+x^{2}_{2}(x^{2}_{0}+x^{2}_{1})+\frac{x^{4}_{2}}{3} \right] j \\
	&\qquad+\left[(x^{2}_{0}x^{2}_{1})+x^{2}_{3}(x^{2}_{0}+x^{2}_{1})+\frac{x^{4}_{3}}{3} \right]ij.
	\end{align*}
\end{example}

We can define a true extension of an analytic function which is left regular and closely resembles the Cauchy-Kowalewski extension found in \cite{Delanghe1,JRyan1}. Again, we are again grateful to Brackx et. al for their proof in the $C\ell_{0,n}$ case, which again gives the convergence of the series.

\begin{theorem}[Cauchy-Kowalewski Extension in $C\ell_{1,1}$]
	Let $g(x_1,x_2,x_3)$ be a $C\ell_{1,1}$-valued function whose components are real-analytic functions on $U\subseteq \R^{3}$. 
	Then the function
	$$f(x_0,x_1,x_2,x_3)=\sum_{k=0}^{\infty}\frac{(-x_0)^k}{k!}D^kg(x_1,x_2,x_3),$$
	where $D=\overline{\partial} -\frac{\partial}{\partial x_0}$, is left-regular in an open neighborhood of $\left\lbrace 0 \right\rbrace \times U$, and $f(0,x_1,x_2,x_3)=g(x_1,x_2,x_3)$.
\end{theorem}

The following lemma will be useful in demonstrating the convergence of $f$ in an open neighborhood of $U$.

\begin{lemma}
	Let $g(x_1,x_2,x_3)$ be a $C\ell_{1,1}$-valued function whose components are real-analytic functions on $U\subseteq \R^{3}$. Then on a compact set $K$, there are constants $c_K$ and $\lambda_K$ such that
	$$\left|D^k g(x_1,x_2,x_3)\right|\leq 3^k c_K(k!)\lambda^{k}_{K}. $$
\end{lemma}

\begin{proof}[Proof of the Lemma]
  Taylor's theorem, again,  gives that on a compact set $K$, there are constants $c_K$ and $\lambda_K$ such that
	$$\left|\frac{\partial^{k}}{\partial x^{k_1}_{1}\partial x^{k_2}_{2}\partial x^{k_3}_{3}} g(x_1,x_2,x_3)\right|\leq c_K(k!)\lambda^{k}_{K}. $$
	It is worth mentioning that we first saw the above inequality in \cite{BrackxDelangheSommen1}.
	
	Notice that when $k$ is even, $D^k$ is a scalar operator. So suppose $k$ is even. Then by the trinomial theorem,
	$$D^k=\sum_{k_1+k_2+k_3=k}\frac{k!}{(k_1)!(k_2)!(k_3)!}\frac{\partial^{k}}{\partial x^{k_1}_{1}\partial x^{k_2}_{2}\partial x^{k_3}_{3}}.$$
	
	Then,
	\begin{align*}
	\left|D^k g(x_1,x_2,x_3)\right| &\leq \sum_{k_1+k_2+k_3=k} \frac{k!}{(k_1)!(k_2)!(k_3)!}\left| \frac{\partial^{k}}{\partial x^{k_1}_{1}\partial x^{k_2}_{2}\partial x^{k_3}_{3}}g(x_1,x_2,x_3)\right|\\
	&\leq  c_K(k!)\lambda_K\sum_{k_1+k_2+k_3=k} \frac{k!}{(k_1)!(k_2)!(k_3)!}\\
	&=3^kc_K(k!)\lambda^{k}_{K}.
	\end{align*}
	
	Now suppose $k$ is odd. Then we have
	\begin{multline*}
	D^k=\sum_{\substack{k_1+k_2+k_3= \\ k-1} }\frac{(k-1)!}{(k_1)!(k_2)!(k_3)!}\left( i\frac{\partial^{k}}{\partial x^{k_1+1}_{1}\partial x^{k_2}_{2}\partial x^{k_3}_{3}}\right. \\\left. -j\frac{\partial^{k}}{\partial x^{k_1}_{1}\partial x^{k_2+1}_{2}\partial x^{k_3}_{3}} 
	- ij\frac{\partial^{k}}{\partial x^{k_1}_{1}\partial x^{k_2}_{2}\partial x^{k_3+1}_{3}}\right).
	\end{multline*}
	
	This means that
	\begin{align*}
	\left| D^kg(x_1,x_2,x_3)\right|&\leq \sum_{\substack{k_1+k_2+k_3= \\ k-1}}\frac{(k-1)!}{(k_1)!(k_2)!(k_3)!}\left| i\frac{\partial^{k}g(x_1,x_2,x_3)}{\partial x^{k_1+1}_{1}\partial x^{k_2}_{2}\partial x^{k_3}_{3}}\right. \\  &\qquad  \left.-j\frac{\partial^{k}g(x_1,x_2,x_3)}{\partial x^{k_1}_{1}\partial x^{k_2+1}_{2}\partial x^{k_3}_{3}} 
	- ij\frac{\partial^{k}g(x_1,x_2,x_3)}{\partial x^{k_1}_{1}\partial x^{k_2}_{2}\partial x^{k_3+1}_{3}}\right|\\
	&\leq\sum_{\substack{k_1+k_2+k_3= \\ k-1}}\frac{(k-1)!}{(k_1)!(k_2)!(k_3)!}\left( \left|\frac{\partial^{k}g(x_1,x_2,x_3)}{\partial x^{k_1+1}_{1}\partial x^{k_2}_{2}\partial x^{k_3}_{3}}\right|\right.   \\  &\qquad \left.  +\left| \frac{\partial^{k}g(x_1,x_2,x_3)}{\partial x^{k_1}_{1}\partial x^{k_2+1}_{2}\partial x^{k_3}_{3}}\right|  
	+\left| \frac{\partial^{k}g(x_1,x_2,x_3)}{\partial x^{k_1}_{1}\partial x^{k_2}_{2}\partial x^{k_3+1}_{3}}\right|\right) \\
	&=\sum_{\substack{k_1+k_2+k_3= \\ k-1}}\frac{(k-1)!}{(k_1)!(k_2)!(k_3)!}\left( 3c_K(k!)\lambda^{k}_{K}\right)\\
	&= 3^kc_K(k!)\lambda^{k}_{K},
	\end{align*}
	as required.		
\end{proof}

\begin{proof}[Proof of the Theorem]
	The above lemma gives us that on a compact set $K\subset U$ there are constants $c_K$ and $\lambda_K$, depending on $K$, such that
	$$\left|\frac{(-x_0)^k}{k!}D^kg(x_1,x_2,x_3)\right|\leq c_K \left(3\lambda_{K}\right)^{k} x^{k}_{0}.$$
	
	Thus, $f$ converges uniformly on
	$$\bigcup_{K\subseteq U}\left[\left(-\frac{1}{3\lambda_K},\frac{1}{3\lambda_K} \right)\times \mathring{K}  \right]. $$

	The essential calculation is 
	\begin{align*}
	\overline{\partial} f &= \sum_{k=0}^{\infty}\overline{\partial}\left( \frac{(-x_0)^k}{k!}D^kg(x_1,x_2,x_3)\right)\\
	&=\sum_{k=0}^{\infty}\overline{\partial}\left( \frac{(-x_0)^k}{k!}\right) D^kg(x_1,x_2,x_3)+\sum_{k=0}^{\infty}\frac{(-x_0)^k}{k!}\overline{\partial} \left( D^kg(x_1,x_2,x_3)\right) \\
	&=-\sum_{k=1}^{\infty} \frac{(-x_0)^{k-1}}{(k-1)!} D^kg(x_1,x_2,x_3)+\sum_{k=0}^{\infty} \frac{(-x_0)^k}{k!} D^{k+1}g(x_1,x_2,x_3)\\
	&=0,
	\end{align*}
	as required.
\end{proof}

\begin{remark}
	We may think of the above extension as a solution to the boundary value problem:
	$$	\begin{cases}
	&\overline{\partial} f(x_0,x_1,x_2,x_3)=0\\
	& f(0,x_1,x_2,x_3)=g(x_1,x_2,x_3)
	\end{cases}.$$
\end{remark}

\begin{example}
	Consider the homogeneous polynomial of degree 2 $$g(x_1,x_2,x_3)=x^{2}_{1}+x^{2}_{2}+x^{2}_{3}+x_1x_2+x_1x_3+x_2x_3.$$
	Now,
	\begin{align*}
	& Dg(x_1,x_2,x_3)=\left(2x_1+x_2+x_3\right)i -\left(2x_2+x_1+x_3 \right)j -\left(2x_3+x_2+x_1 \right) ij\\
	& D^2g(x_1,x_2,x_3)=-6\\
	& D^3(x_1,x_2,x_3)=0.
	\end{align*}
	Thus,
	\begin{align*}
	f(Z)=&\left(-3x^{2}_{0}+x^{2}_{1}+x^{2}_{2}+x^{2}_{3}+x_1x_2+x_1x_3+x_2x_3 \right)\\
	&-x_0\left(2x_1+x_2+x_3 \right)i-x_0\left(2x_2+x_1+x_3 \right)j-x_0\left(2x_3+x_1+x_2 \right)ij.
	\end{align*}
	is the left- regular function obtained by the above theorem.
\end{example}

\begin{remark}
	In both of these formulas, a polynomial $g$ will be transformed to a Clifford valued function where every component is a polynomial. This is the case because polynomials have partial derivatives of $0$ after a certain order. That is, $D^k g$ and $\Delta^k g$ will be uniformly $0$ for all $k>M$ for some finite $M$.
\end{remark}

	\bibliographystyle{plain}
	\bibliography{EmanuelloNolderBibliography2}
\end{document}